\documentclass[a4paper,12pt]{amsart}
\usepackage{amsmath,mathrsfs,amssymb,xcolor,MnSymbol}
\usepackage[top=3cm,bottom=3cm,outer=3cm,inner=3cm,marginpar=3cm]{geometry}
\usepackage[utf8]{inputenc}

\usepackage{hyperref}
\hypersetup{
	colorlinks = true,
	linkcolor = red,
	anchorcolor = black,
  	citecolor = blue,
	filecolor = red,
	urlcolor = magenta,
	pdfauthor=author
	menucolor = red
}

\newtheorem{theorem}{Theorem}[section]

\newtheorem{lemma}[theorem]{Lemma}

\newtheorem*{corollary*}{Corollary}
\newtheorem{proposition}[theorem]{Proposition}

\theoremstyle{definition}

\newtheorem{remark}[theorem]{Remark}

\newtheorem{question}[theorem]{Quesion}

\numberwithin{equation}{section}

\newcommand{\paren}[1]{\left(#1\right)}
\newcommand{\bparen}[1]{\left[#1\right]}

\newcommand{\abs}[1]{\left\vert#1\right\vert}
\newcommand{\norm}[1]{\left\|#1\right\|}

\newcommand{\inner}[1]{\left\langle{#1}\right\rangle}

\newcommand{\set}[1]{\left\{#1\right\}}

\newcommand{\ip}[1]{\mathrm{Im}\;s}

\newcommand{\iinner}[1]{\left\llangle{#1}\right\rrangle}
\newcommand{\rd}[1]{\left\lfloor{#1}\right\rfloor}
\newcommand{\ru}[1]{\left\lceil{#1}\right\rceil}
\newcommand{\Ltwo}[1]{L^{#1}_{(2)}}

\def\ol#1{{\overline{#1}}}

\def\ii{\mathbf{i}}

\def\we{\wedge}

\def\im{\mathrm{Im}}

\def\dbar{\overline\partial}

\def\RR{\mathbb{R}} 
\def\CC{\mathbb{C}} 
\def\DD{\mathbb{D}} 
\def\dbar{\ol\partial}

\def\ric{\mathrm{Ric}}

\def\del{\partial}
\def\we{\wedge}
\def\dbar{\overline\partial}

\def\sK{\mathsf{K}}
\def\sL{\mathsf{L}}
\def\dom{\mathrm{Dom}\,}
\def\harm{\mathcal H}
\def\Ltwoharm{\harm_{(2)}}
\def\dc{d^c}

\begin{document}
\title[A Sharp Lower Bounds for the Spectrum]{A Sharp Lower Bound for the Spectrum of the Hodge Laplacian on K\"ahler Hyperbolic Manifolds and its Applications}

\author{Ye-Won Luke Cho}
\address{Research Institute of Molecular Alchemy, Gyeongsang National University,
	Jinju,  Gyeongnam, 52828, Republic of Korea}
\email{ywlcho@gnu.ac.kr}

\author{Young-Jun Choi}
\address{Department of Mathematics and Institute of Mathematical Science, Pusan National University, 2, Busandaehak-ro 63beon-gil, Geumjeong-gu, Busan, 46241, Republic of Korea}
\email{youngjun.choi@pusan.ac.kr}

\author{Kang-Hyurk Lee}
\address{Department of Mathematics and Research Institute of Natural Science, Gyeongsang National University, Jinju, Gyeongnam, 52828, Republic of Korea}
\email{nyawoo@gnu.ac.kr}


\thanks{The first named author was supported by G-LAMP (RS-2023-00301974). The second named author was supported by the National Research Foundation of Korea (NRF) grant funded by the Korea government (No.~NRF-2023R1A2C1007227). 
The second and third named authors were supported by Samsung Science and Technology Foundation under Project Number SSTF-BA2201-01. }

\begin{abstract}
In this paper, we establish a sharp lower bound for the spectrum of the Hodge Laplacian on K\"ahler hyperbolic manifolds. 
This bound is expressed explicitly in terms of the supremum norm of the $1$-form associated with the K\"ahler hyperbolic structure. 
As an application, we obtain explicit spectral lower bounds for bounded symmetric domains.
\end{abstract}

\maketitle

\section{Introduction}

Let $X$ be an $n$-dimensional complete K\"ahler manifold and $\omega$ be its K\"ahler form. We say that $\omega$ is \emph{$d$-bounded} if there exists a real $1$-form $\eta$ on $X$ such that 
\begin{equation*}
	d\eta=\omega\quad\text{and}\quad \norm{\eta}_{L^\infty}
	:=
	\sup_X \abs{\eta}_\omega<+\infty
\end{equation*}
where $\abs{\eta}_\omega$ denotes the pointwise length of $\eta$ measured by $\omega$. In \cite{Gromov1991}, Gromov proved that if $X$ admits a $d$-bounded K\"ahler form $\omega$, then 
\begin{equation}\label{thm:Lefschetz_vanishing_theorem}
\dim\mathcal H^{p,q}_{(2)}(X) = 0 \quad \text{for } p+q \neq n,
\end{equation}
where $\mathcal H^{p,q}_{(2)}(X)$ is the space  of square integrable harmonic $(p,q)$-forms with respect to $\omega$. Moreover, he also proved that if $X$ is simply-connected and covers a compact K\"ahler manifold with $\omega$ being a lifting of a K\"ahler form of the quotient, then 
\begin{equation*}
\dim\mathcal H^{p,q}_{(2)}(X)=\infty \quad \text{for } p+q=n.
\end{equation*}

Prior to this, Donnelly and Fefferman had proved the vanishing theorem for the $L^2$-cohomology of bounded strongly pseudoconvex domains equipped with the Bergman metric \cite{Donnelly_Fefferman1983}. 
Subsequently, Donnelly utilized Gromov's technique to provide an alternative proof of the Donnelly-Fefferman theorem and extended the result to certain classes of weakly pseudoconvex domains \cite{Donnelly1994}.

In \cite{Gromov1991}, the vanishing theorem~\eqref{thm:Lefschetz_vanishing_theorem} (which is called the Lefschetz vanishing theorem in \cite{Gromov1991}) was proved by using the $L^2$-Lefschetz isomorphism theorem. 
More precisely, let $\varphi$ be an $L^2$-harmonic $k$-form with $k<n$. 
Then $\omega^{n-k}\wedge\varphi$ is an $L^2$-harmonic $(2n-k)$-form, and it satisfies
\begin{equation}
\label{eqn:Lefschetz_map}
    \omega^{n-k}\wedge\varphi
    =
    d\paren{\eta\wedge\omega^{n-k-1}\wedge\varphi},
\end{equation}
Moreover, $\eta\wedge\omega^{n-k-1}\wedge\varphi$ belongs to $L^2$, since the $L^\infty$-norm of $\eta$ is bounded.
Since the Lefschetz map $L^k:\varphi\mapsto\omega^{n-k}\wedge\varphi$ is bijective on $k$-forms, it follows that $\varphi$ vanishes.

On the other hand, Gromov also sharpened the proof of the vanishing theorem, which means that he estimated a lower bound for the spectrum of the Hodge Laplacian $\Delta:=\Delta_d:=dd^*+d^*d$.
More precisely,
\begin{theorem}[\cite{Gromov1991}]
\label{thm:gromov}
    Let $(X,\omega)$ be a complete K\"ahler manifold of dimension $n$ and $\omega=d\eta$ where $\eta$ is a bounded $1$-form on $X$. 
    Then every $k$-form $\varphi\in\dom(\Delta)\cap\Ltwo{k}(X)$ with $k\neq n$ satisfies the inequality
    \begin{equation*}
        \iinner{\Delta\varphi,\varphi}
        \ge
        \frac{c_k}{\norm{\eta}^2_{L^\infty}}
        \norm{\varphi}^2,
    \end{equation*}
    where $c_k\geq 0$ is a constant depending on $k$ and $n$.
\end{theorem}
Here, $\iinner{\,\cdot\,,\,\cdot\,}$ and $\norm{\,\cdot\,}$ stand for the $L^2$ inner product and norm with respect to the volume form $dV_\omega$ of $\omega$, and
\begin{equation*}
	\dom(\Delta)=\set{\varphi\in\Ltwo{\bullet}:\Delta\varphi\in\Ltwo{\bullet}},
\end{equation*}
where $\Ltwo{\bullet}$ denotes the space of measurable $\bullet$-forms with finite $L^2$-norm.

In this paper, we first investigate the constant $c_k$ explicitly. 
More precisely, we consider the Laplacian acting on $(p,q)$-forms and, through careful computations, determine an explicit constant $c_{p,q}$, from which an explicit expression for $c_k$ follows.

In what follows, a complete Kähler manifold is called \emph{K\"ahler hyperbolic} if its K\"ahler form is $d$-bounded (see~\cite{Gromov1991} for the original definition of K\"ahler hyperbolicity).
\begin{theorem}
    \label{thm:main_theorem1}
    Let $(X,\omega)$ be a K\"ahler hyperbolic manifold with a bounded $1$-form $\eta$ such that $d\eta=\omega$.
		Let $(p,q)$ be a pair of nonnegative integers satisfying $k:=p+q\neq n$. Then
    \begin{equation*}
    	\iinner{\Delta\varphi,\varphi}
    	\ge
    	\frac{c_{p, q}}{\norm{\eta}^2_{L^\infty}}
    	\norm{\varphi}^2,\quad\text{for all}\;\;\varphi\in\dom(\Delta)\cap\Ltwo{p,q}(X)
    \end{equation*}
    where
    \begin{equation*}
    	c_{p, q}:=  
    	\begin{cases}
    		\frac{(n-k)!^4}{4}
    		\cdot
    		\frac{(p+1)!^4}{(n-q)!^4}~\textit{if}~k< n,\\
    		c_{n-p,n-q}~\textit{if}~k>n.
    	\end{cases}     
    \end{equation*}
    This in particular implies that 
    \begin{equation*}
        \iinner{\Delta\varphi,\varphi}
        \ge
        \frac{c_k}{\norm{\eta}^2_{L^\infty}}
        \norm{\varphi}^2\quad\text{for all}\;\;\varphi\in\dom(\Delta)\cap\Ltwo{k}(X)
\end{equation*}
    where
    \begin{equation*}
        c_k:= \begin{cases}
        	\frac{(n-k)!^4}{4}
        \cdot
        \frac{(\ru{k/2}+1)!^4}{(n-\rd{k/2})!^4}~\textit{if}~k<n,\\
        c_{2n-k}~\textit{if}~k>n.
        \end{cases}
    \end{equation*}
\end{theorem}

As in the previous argument, the lower bound estimate begins with the Lefschetz map $L^{n-k}$, which is actually quasi-isometry from the space of $k$-forms to the space of $(2n-k)$-forms. 
Via this map, the norm of $\varphi$ can be estimated in terms of $\omega^{n-k}\wedge\varphi$, which is decomposed as
	\begin{align*}
		L^{n-k}\varphi
		&=
		d\paren{\varphi\wedge\eta\wedge\omega^{n-k-1}}
		-
		d\varphi\wedge\eta\wedge\omega^{n-k-1}.
	\end{align*}
The argument is then completed by invoking the Lefschetz decomposition theorem, which enables us to derive a quantitative lower bound (see Section~\ref{sec:proof_of_main_theorems}).

In the middle-dimensional case $k=n$, no vanishing theorem holds in general so nontrivial $L^2$-harmonic $n$-forms may exist.
As a result, there is no positive lower bound for spectrum of $\Delta$ in this case.
However, by restricting $\Delta$ to the orthogonal complement of the space of $L^2$-harmonic $n$-forms, one can still obtain a spectral lower bound. 
This is discussed in Section~\ref{sec:middle_dimension}.

In the case $k=0$, corresponding to the Laplacian acting on functions, one can obtain a sharper lower bound than in the general case (see Remark~\ref{rmk:better_estimate}).
We begin by recalling the definition of \emph{the bottom of the spectrum of the Laplacian $\Delta$}. 
It is given by
 \begin{equation}\label{eqn:Rayleigh}
    \lambda_0(X)
    :=
    \inf
		\set{
			\int_X\abs{d\varphi}_\omega^2 dV_\omega
			:\varphi\in C_0^\infty(X),\; \int_X\abs{\varphi}^2dV_\omega=1
		}.
\end{equation}

\begin{theorem}\label{thm:main_theorem2}
If $\varphi\in \dom(\Delta)\cap\Ltwo{0}(X)$, then
\begin{equation}\label{eq:lower_bound_estimate_functions}
    \iinner{\Delta\varphi,\varphi}
    \ge
    \frac{n^2}{4\norm{\eta}^2_{L^\infty}}\norm{\varphi}^2.
\end{equation}
In particular, 
\begin{equation*}
\lambda_0(X)\ge\frac{n^2}{4\norm{\eta}^2_{L^\infty}}.
\end{equation*}
\end{theorem}

It is remarkable that this estimate in Theorem~\ref{thm:main_theorem2} is optimal, in the sense that the equality holds when $X$ is the unit ball $\mathbb{B}^n$.
More precisely, by Theorem 1.1 in \cite{Choi_Lee_Seo2025Ar}, there exists a potential $\varphi$ of the complete K\"ahler-Einstein metric $\omega$ of Ricci curvature $-(n+1)<0$ (i.e., $\ric(\omega)=-(n+1)\omega$) on the unit ball $\mathbb{B}^n$ in $\CC^n$ such that
\begin{equation*}
    \abs{\partial\varphi}^2_\omega=1
    \quad\text{equivalently}\quad
    \abs{d^c\varphi}^2_\omega=\frac{1}{2}.
\end{equation*}
By putting $\eta=d^c\varphi$, Theorem~\ref{thm:main_theorem2} says 
\begin{equation*}
    \lambda_0(\mathbb{B}^n)
    \ge
    \frac{n^2}{2}.
\end{equation*}
But it is known that $\lambda_0(\mathbb{B}^n)=n^2/2$ under the normalization $\ric(\omega)=-(n+1)\omega$ (see, e.g., \cite{Chavel-Book,McKean1970}).

When $X$ is a compact manifold, the spectrum of the Laplacian $\Delta$ is discrete and satisfies
\begin{equation*}
0=\lambda_0(X)<\lambda_1(X)\le\lambda_2(X)\le\cdots.
\end{equation*}
In this case, the first positive eigenvalue $\lambda_1(X)$ plays a fundamental role in geometry.
Its upper and lower bounds and their relation to the underlying geometry have been extensively studied; see, for example, \cite{Lichnerowicz1958}, \cite{Obata1962}, \cite{Cheeger1970}, \cite{Cheng1975}, \cite{Chavel-Book}, \cite{Udagawa1988}, \cite{Li-LN}, among many others.

In contrast, when $X$ is noncompact, the spectrum need not be discrete, and $\lambda_0(X)$, defined in~\eqref{eqn:Rayleigh}, may fail to be an eigenvalue of $\Delta$, but rather the infimum of the spectrum of $\Delta$.
Nevertheless, it plays a role analogous to that of the first positive eigenvalue in the compact case. The problem of estimating $\lambda_0(X)$ in this setting has therefore attracted significant attention.
An important upper bound was obtained by Li and Wang~\cite{Li_Wang2005}.
With the assumption that the holomorphic bisectional curvature of a K\"{a}hler manifold $(X,\omega)$ is bounded from below by $-1$, they proved that $\lambda_0(X)\ge n^2/2$.
Later, Munteanu~\cite{Munteanu2009} proved that $\lambda_0(X) \leq n^2/2$ under the assumption that the Ricci curvature is bounded from below by $-(n+1)$, i.e., $\operatorname{Ric}(\omega) \geq -(n+1)\omega$. 
Both estimates are sharp in the sense that the equality is achieved on the complex hyperbolic space $(\mathbb{B}^n,\omega)$ as above so that $\lambda_0(\mathbb{B}^n) = n^2/2$.

Regarding lower bounds, McKean~\cite{McKean1970} established a result for complete Riemannian manifolds under the assumption that the sectional curvature is bounded above by a negative constant.
Subsequently, Setti~\cite{Setti1991} derived a sharper lower bound involving both sectional and Ricci curvatures.
In the K\"ahler setting, to the best of the authors’ knowledge, no lower bound for $\lambda_0$ is currently known under assumptions involving only the holomorphic sectional curvature or the Ricci curvature. 
This leads to the following question.

\begin{question}
Is there a lower bound for the bottom of the spectrum of a K\"ahler manifold that depends only on the holomorphic sectional curvature?\end{question}

The following theorem gives an affirmative answer to this question in the case of bounded symmetric domains.

\begin{theorem}\label{thm:main_cor}
    If $(\Omega,\omega)$ is a bounded symmetric domain equipped with the complete K\"ahler-Einstein metric $\omega$ whose holomorphic sectional curvature is bounded from above by $-K<0$, then 
    \begin{equation*}
        \lambda_0(\Omega)
        \ge
        \frac{n^2}{4}\cdot K.
    \end{equation*}
\end{theorem}

Note that Theorem~\ref{thm:main_cor} does not follow from \cite{McKean1970} or \cite{Setti1991}, since any bounded symmetric domain other than the unit ball admits directions along which the sectional curvature vanishes.
In~\cite{Choi_Lee_Seo2025Ar}, the authors obtained a sharp upper bound for the holomorphic sectional curvature of bounded symmetric domains equipped with the K\"ahler-Einstein metrics by introducing the K\"ahler-hyperbolicity length which is defined in terms of the rank and the generic norm.
Combined with this result, Theorem~\ref{thm:main_cor} provides a sharp lower bound for the bottom of the spectrum on irreducible bounded symmetric domains classified by \'E.~Cartan (see Section~\ref{subsec:irreducible_domains}).
\medskip

This paper is organized as follows. Section~\ref{sec:Preliminaries} reviews basic notions and results in K\"ahler geometry, including the primitive decomposition and $L^2$-Hodge theory. 
In Section~\ref{sec:Formula}, several pointwise identities which provide the key ingredients for the proofs of the main theorems are established.
Section~\ref{sec:proof_of_main_theorems} presents the proofs of Theorem~\ref{thm:main_theorem2} and Theorem~\ref{thm:main_theorem1}. 
Section~\ref{sec:middle_dimension} focuses on the bottom of the spectrum in the middle-dimensional case. 
Finally, in Section~\ref{sec:Spectrum_BSD}, we consider the bottom of the spectrum of bounded symmetric domains, proving Theorem~\ref{thm:main_cor} and obtaining explicit lower bounds for irreducible bounded symmetric domains.
\tableofcontents

\section{Preliminaries}
\label{sec:Preliminaries}
In this section, we review some basic terminology and results in K\"ahler geometry and also recapitulate the essential results of $L^2$-Hodge theory to be used throughout this paper. 
We refer the reader to \cite{Demailly-Book, Huybrechts-Book} for details.

\subsection{Basic hermitian geometry}
\label{subsec:basic_results}
	Let \( X \) be a complex manifold of complex dimension \( n \) and let $T_\CC^*X$ be the complexified cotangent bundle on $X$.
	The complexified exterior algebra $\bigwedge^\bullet T_\CC^*X:=\bigoplus_k\bigwedge^kT_\CC^*X$ admits a natural decomposition by type:
	\begin{equation}
	\label{eqn:pq_decomposition}
		\bigwedge\nolimits^{\!k}T_\CC^*X
		=
		\bigoplus_{p,q=k}\bigwedge\nolimits^{\!p,q}T^*X,
		\quad\text{for}\quad k=0,\cdots,2n,
	\end{equation}
	where $\bigwedge^{p,q}T^*X:=\bigwedge^p(T^*X)^{1,0}\otimes\bigwedge^q(T^*X)^{0,1}$.
	
	Suppose that $X$ admits a hermitian metric $g$ on $X$. 
	Then the associated K\"ahler form $\omega:=\omega_g$ is written as
	\[
	\omega = \ii \sum_{\alpha, \beta} g_{\alpha \bar{\beta}} \, dz^\alpha \wedge d\bar{z}^\beta,
	\]
	under a local holomorphic coordinates $(z^1,\ldots,z^n)$.
	It induces a canonical hermitian (sesquilinear) product on each $\bigwedge\nolimits^{\!p,q}T^*X$ so on $\bigwedge\nolimits^{\!k}T_\CC^*X$, which we denote by $\inner{\cdot,\cdot}_\omega$, with the associate pointwise norm $\abs{\cdot}_\omega$.
    The corresponding global inner product (i.e., $L^2$-inner product with respect to $dV_\omega$) and its induced norm are denoted by $\iinner{\cdot, \cdot}_\omega$ and $\norm\cdot_\omega$, respectively.
We shall omit the subscript $\omega$, if there is no confusion.
	
	The Hodge star operator $*:\bigwedge^kT_\CC^*X\rightarrow\bigwedge^{2n-k}T_\CC^*X$ is defined by the equation
    \begin{equation*}
    	\varphi\wedge *\ol\psi
		=
		\inner{\varphi,\psi}_\omega dV_\omega
    \end{equation*}
    for all $\varphi,\psi\in\bigwedge^\bullet T^*_\CC X$, where the volume form is given by $dV_\omega:=\omega^n/n!$.
    The Lefschetz operator $L:\bigwedge^{p,q}T^*X\rightarrow\bigwedge^{p+1,q+1}T^*X$ is defined by
    \begin{equation*}
    	L\varphi=\omega\wedge\varphi,
    \end{equation*}
    and the dual Lefschetz operator $\Lambda:\bigwedge^{p,q}T^*X\rightarrow\bigwedge^{p-1,q-1}T^*X$ is the adjoint operator of $L$ with respect to $\inner{\cdot,\cdot}_\omega$, which is, in fact, given as $\Lambda=*^{-1}\circ L\circ*$.

Let $\Omega^k(X)$ and $\Omega^{p,q}(X)$ denote the spaces of $k$-forms and $(p,q)$-forms, respectively. 
The exterior differentiation 
$$
d:\Omega^k(X)\rightarrow\Omega^{k+1}(X)
$$ 
has the formal adjoint operator $d^*$ is given by 
$$
d^*=-*\circ d\circ* : \Omega^{k+1}(X)\rightarrow\Omega^{k}(X).
$$ 
The second-order elliptic operator
\begin{equation*}
\Delta:=dd^*+d^* d:\Omega^k(X)\to \Omega^k(X),
\end{equation*}
is called the $\textit{Hodge Laplacian}$.
Recall that
$d=\partial+\bar{\partial}$ on each $\Omega^{p,q}(X)$ (hence on each $\Omega^{k}(X)$).
Then
\begin{equation*}
	\del^*
	:=
	-*\circ \bar{\del}\circ*:\Omega^{p+1,q}(X)\to\Omega^{p,q}(X),
	\quad\bar{\del}^*
	=
	-*\circ \del\circ*:\Omega^{p,q+1}(X)\to\Omega^{p,q}(X)
\end{equation*}
are the formal adjoints of $\del$ and $\dbar$ with respect to $\iinner{\cdot,\cdot}_\omega$, respectively.
The Laplacians on $\Omega^{p,q}(X)$ associated with the operators $\del, \bar{\del}$ are 
\begin{equation*}
\Delta_\del=\del\del^*+\del^*\del,
\quad
\Delta_{\dbar}=\bar{\del}\,\bar{\del}^*+\bar{\del}^*\bar{\del}.
\end{equation*}
If $(X,\omega)$ is K\"ahler, then we have the following fundamental identities.
\begin{equation}\label{eq:Kahler_identity}
\Delta_\del=\Delta_{\dbar}=\frac{1}{2}\Delta,
\end{equation}
and $\Delta$ commutes with $*,\del,\dbar,\del^*,\dbar^*, L,$ and $\Lambda$.
\bigskip

An element $\varphi\in\bigwedge^k T^*X$ is called primitive if $\Lambda\varphi=0$.
The linear subspace of all primitive elements $\varphi\in\bigwedge^k T^*_\CC X$ is denoted by $P_\CC^kX\subset\bigwedge^k T_\CC^*X$.
Then every element of $\bigwedge^k T_\CC^*X$ is decomposed into primitive elements as follows.
\begin{proposition}
[Proposition 1.2.30 in \cite{Huybrechts-Book}]
\label{prop:Lefschetz_decomposition}        
We have the following.
\begin{enumerate}
		\item There exists a direct sum decomposition of the form:
		\begin{equation*}
				\bigwedge\nolimits^{\!k}T_\CC^*X
				=
				\bigoplus_{i\ge0}L^r \paren{P_\CC^{k-2i}X}.
		\end{equation*}
This is the Lefschetz decomposition.
Moreover, this decomposition is orthogonal with respect to $\inner{\cdot,\cdot}_\omega$.
		\item If $k>n$, then $P_\CC^kX=0$.
		\item The map $L^{n-k}:P_\CC^k\rightarrow\bigwedge^{2n-k}T_\CC^*X$ is injective for $k\le n$.
		\item The map $L^{n-k}:\bigwedge^{k}T_\CC^*X\rightarrow\bigwedge^{2n-k}T_\CC^*X$ is bijective for $k\le n$.
		\item If $k\le n$, then 
		$P_\CC^k=\set{u\subset\bigwedge^kT_\CC^kX:L^{n-k+1}u=0}$.
\end{enumerate}
\end{proposition}
	
The following propositions are crucial for the computation of the hermitian inner products on the exterior tensor bundles.
	
\begin{proposition}[Proposition 1.2.31 in \cite{Huybrechts-Book}]
\label{prop:star of L}
	For $\varphi\in P_\CC^kX$, one has
	\begin{equation*}
			*L^r\varphi
			=
			\ii^{k(k+1)}
			\frac{r!}{(n-k-r)!}
			\cdot
			L^{n-k-r}\mathbf{I}(\varphi),
	\end{equation*}
	where
	\begin{equation*}
			\mathbf{I}=\sum\ii^{p-q}\Pi_{p,q},
	\end{equation*}
	and $\Pi_{p,q}$ denotes the projection onto the $(p,q)$-component with respect to the decomposition~\eqref{eqn:pq_decomposition}.
\end{proposition}
    
The Hodge-Riemann pairing $Q:\bigwedge^k T^*X\times\bigwedge^kT^*X\rightarrow\bigwedge^{2n}T^*X$ is defined by
\begin{equation*}
	Q(\varphi,\psi)=\ii^{k(k-1)}\omega^{n-k}\wedge\varphi\wedge\psi.
\end{equation*}
This pairing can be extended $\CC$-linearly to $\bigwedge^\bullet T_\CC^*X$, which is again denoted by $Q$.
The following formula is known as the Hodge-Riemann bilinear relation. 
\begin{proposition}[Corollary 1.2.36 in \cite{Huybrechts-Book}]
	For $\alpha,\beta\in P_\CC^{p,q}X$, one has
	\begin{equation*}
		\ii^{p-q}Q(\alpha,\ol\beta)
		=
		(n-(p+q))!\inner{\alpha,\beta}_\omega dV_\omega.
	\end{equation*}
\end{proposition}

\subsection{$L^2$-Hodge theory}
Let $(X,g)$ be a Riemannian manifold.
As in Section~\ref{subsec:basic_results}, the metric $g$ induces an inner product $\iinner{\cdot,\cdot}_g$ on $\Omega^k(X)$ for each integer $k\geq 0$ and 
\begin{equation*}
 \Ltwo{k}(X):=\set{u:u~\text{is a $k$-form on $X$ with measurable coefficients and}~\norm{u}_{g}<+\infty},
\end{equation*}
equipped with $\iinner{\cdot,\cdot}_g$ becomes a Hilbert space. 
Since the set $\mathcal D^k(X)$ of smooth $k$-forms on $X$ with compact support is dense in $\Ltwo{k}(X)$, the operators
\begin{equation*}
d:\Omega^k(X)\rightarrow \Omega^{k+1}(X),~d^*: \Omega^{k}(X)\rightarrow \Omega^{k-1}(X),~\Delta:\Omega^k(X)\rightarrow\Omega^k(X)
\end{equation*}
can be extended in the sense of distribution to closed, densely defined, unbounded operators
\begin{equation*}
    d:\Ltwo{k}(X)\rightarrow \Ltwo{k+1}(X),~d^*:\Ltwo{k}(X)\rightarrow \Ltwo{k-1}(X),~\Delta:\Ltwo{k}(X)\rightarrow\Ltwo{k}(X),
\end{equation*}
respectively. We denote by  $\dom d$, $\dom d^*$, $\dom \Delta$ the domains of the extended operators. We also denote by $\textup{Im}\,d$, $\textup{Im}\,d^*$ the images of $d$ and $d^*$, respectively.

The celebrated $L^2$-Hodge theory is as follows.
\begin{theorem}[Ch.VIII (3.2) Theorem in \cite{Demailly-Book}]\label{thm:Hodge_theory}
    If $(X,g)$ is a complete Riemannian manifold, then the following holds.
    \begin{enumerate}
        \item The space of smooth $k$-forms with compact support $\mathcal D^k(X)$ is dense in $\dom d$, $\dom d^*$ and $\dom d\cap\dom d^*$ respectively for the graph norms
        \begin{equation*}
            u\mapsto\norm u_g+\norm{du}_g,
            \quad
            u\mapsto\norm u_g+\norm{d^* u}_g,
            \quad
            u\mapsto\norm u_g+\norm{du}+\norm{d^* u}_g.
        \end{equation*}
        \item $d_\harm^*=d^*, d_\harm^{**}=d$ as Hilbert adjoint operators where $d_\harm^*$ is the Hilbert \textup{(}von Neumann\textup{)} adjoint.
        \item One has $\iinner{u,\Delta u}_g=\norm{du}_g^2+\norm{d^* u}_g^2$ for every $u\in\dom\Delta$.
        In particular,
        \begin{equation*}
            \dom\Delta\subset\dom d\cap\dom d^*,
            \quad
            \ker\Delta=\ker d\cap\ker d^*,
        \end{equation*}
        and $\Delta$ is self-adjoint.
        \item There are orthogonal decompositions
        \begin{eqnarray*}
            L^k_{(2)}(X)&=&\mathcal H^k_{(2)}(X)
            \oplus\ol{\im\; d}\oplus\ol{\im\;d^*}
            \\
            \ker d&=&\mathcal H^k_{(2)}(X)\oplus\ol{\im\; d}
        \end{eqnarray*}
        where $\mathcal H^k_{(2)}(X)=\set{u\in L^k_{(2)}(X):\Delta u=0}\subset\Omega^k(X)$ is the space of $L^2$-harmonic $k$-forms.
    \end{enumerate}
\end{theorem}

\begin{remark}\label{rmk:bottom of the spectrum}
We define the bottom of the spectrum $\lambda_0^k(X)$ of $\Delta$ on $\Omega^k(X)$ by
\[
\lambda_0^k(X)= \inf\set{\norm{du}^2_g+\norm{d^* u}^2_g:u\in\mathcal{D}^k(X),\,\norm{u}_g =1}.
\]
Then Theorem~\ref{thm:Hodge_theory} implies that
\begin{equation*}
	\iinner{\Delta\varphi,\varphi}
        \ge
        \lambda_0^k(X)
        \norm{\varphi}^2\quad\text{for all}\;\;\varphi \in \dom d.
\end{equation*}
Hence Theorem~\ref{thm:main_theorem1} yields that
\begin{equation*}
	\lambda_0^k(X)\ge c_k.
\end{equation*}
The bottom of the spectrum $\lambda_0^{p,q}(X)$ is defined analogously, and the same conclusion holds in this case as well.
\end{remark}


\begin{remark}\label{rmk:Hodge_theory}
If $(X,\omega)$ is a complete hermitian manifold, then analogous results hold for $\del$ and $\dbar$, replacing $k$-forms with $(p,q)$-forms. 
More precisely, 
if we denote by $\Ltwo{p,q}(X)$ the set of $(p,q)$-forms on $X$ with measurable coefficients and finite $L^2$-norm, then the set $\Ltwo{p,q}(X)$ equipped with $\iinner{\cdot,\cdot}_g$ is a Hilbert spaces for each  $(p,q)$. 
Since the set  $\mathcal D^{p,q}(X)$ of smooth $(p,q)$-forms on $X$ with compact support is dense in $\Ltwo{p,q}(X)$, the operators
	\[
	\quad\bar{\partial}:\Omega^{p,q}(X)\to\Omega^{p,q+1}(X),
	~\dbar^{\ast}:\Omega^{p,q}(X)\to\Omega^{p,q-1}(X),
	~\Delta_{\bar{\partial}}:\Omega^{p,q}(X)\to\Omega^{p,q}(X)
	\]
	extend in the sense of distribution to closed, densely defined, unbounded operators
	\begin{equation*}
		\quad\bar{\partial}:\Ltwo{p,q}(X)\rightarrow \Ltwo{p,q+1}(X),
		~\bar{\partial}^{\ast}:\Ltwo{p,q}(X)\rightarrow \Ltwo{p,q-1}(X),
		~\Delta_{\bar{\partial}}:\Ltwo{p,q}(X)\rightarrow \Ltwo{p,q}(X),
	\end{equation*}
respectively. Then it is known that Theorem  \ref{thm:Hodge_theory} also holds for the extended operators $\bar{\partial}$, $\bar{\partial}^{\ast}$, $\Delta_{\bar{\partial}}$. 
	
Moreover, if $(X,\omega)$ is K\"ahler, then harmonic forms with respect to $d,\del,\dbar$ coincide by \eqref{eq:Kahler_identity}.
This leads to the Hodge decomposition of differential forms:
\begin{equation*}
\Ltwoharm^k(X)=\bigoplus_{p+q=k}\Ltwoharm^{p,q}(X),
\end{equation*}
where $\harm^{p,q}_{(2)}(X):=\set{u\in \dom\Delta_{\bar{\partial}}:\Delta_{\bar{\partial}} u=0,~\norm{u}_g<+\infty}$.
For further details, readers are referred to \cite{Demailly-Book, Huybrechts-Book}.
\end{remark}

\section{Auxiliary Formulae}
\label{sec:Formula}

In this section, we derive several formulae concerning the norms of differential forms.
These results will play a crucial role in the proofs presented in the subsequent section.

\begin{proposition}\label{prop:inner_product_Lefschetz}
If $\varphi,\psi\in P^kX$ and $k\le n$, then
	\begin{equation*}
		\inner{L^j\varphi,L^j\psi}_\omega
		=
		\frac{j!(n-k)!}{(n-k-j)!}\inner{\varphi,\psi}_\omega,
	\end{equation*}
    for any $j\in \{0,\dots, n-k\}.$
    In particular, if $j=n-k$ and $j=n-k-1$, then
    \begin{equation*}
		\inner{L^{n-k}\varphi,L^{n-k}\psi}_\omega
		=
		(n-k)!^2\inner{\varphi,\psi}_\omega,
	\end{equation*}
	and
	\begin{equation*}
		\inner{L^{n-k-1}\varphi,L^{n-k-1}\psi}_\omega
		=
		(n-k-1)!(n-k)!\inner{\varphi,\psi}_\omega,
	\end{equation*}
	respectively.
\end{proposition}

Note that, if $j\ge n-k+1$, then $L^j\varphi=L^j\psi=0$ by $(5)$ in Proposition~\ref{prop:Lefschetz_decomposition}.

\begin{proof}
	Since the bidegree decomposition is orthogonal with respect to $\inner{\cdot,\cdot}_\omega$, $\varphi$ and $\psi$ can be assumed to be of type $(p,q)$ where $p+q=k$. Note first that
	\begin{equation*}
		\inner{L^j\varphi,L^j\psi}dV_\omega
		=
		L^j\varphi\wedge\ol{*L^j\psi}
	\end{equation*}
	by the definition of $\ast$. Since $\psi$ is primitive, Proposition~\ref{prop:star of L} implies that
	\begin{equation*}
		*L^j\psi
		=
		\frac{j!}{(n-k-j)!}\ii^{k(k+1)+p-q}L^{n-k-j}\psi.
	\end{equation*}
	It follows from the Hodge-Riemann bilinear relation that
	\begin{align*}
		\inner{L^j\varphi,L^j\psi} dV_\omega
		&=
		\frac{j!}{(n-k-j)!}
        L^j\varphi\wedge\ol{\ii^{k(k+1)+p-q}L^{n-k-j}\psi}
		\\
		&=
		\frac{j!}{(n-k-j)!}
		(-\ii)^{k(k+1)+p-q}
		\omega^{n-k}\wedge\varphi
		\wedge\ol{\psi}
		\\
		&=
		\frac{j!}{(n-k-j)!}
		\ii^{k(k-1)}\cdot \ii^{p-q}
		\omega^{n-k}\wedge\varphi
		\wedge\ol{\psi}
		\\
		&=
		\frac{j!}{(n-k-j)!}
		\ii^{p-q}Q\paren{\varphi,\ol{\psi}}
		=
		\frac{j!(n-k)!}{(n-k-j)!}
		\inner{\varphi,\psi} dV_\omega.
	\end{align*}
	This completes the proof.
\end{proof}

\begin{lemma}
    \label{lem:pnorm_primitive_decomposition}
    Let $\varphi, \psi\in \Lambda^{k}X$ with $k<n$. If $\varphi=\sum_rL^r\varphi_r,$ $\psi=\sum_rL^r\psi_r,$  are the primitive decompositions of $\varphi$ and $\psi$, then we have the following identities:
    \begin{align*}
       \inner{\varphi,\psi}_{\omega}
        &=
        \sum_r\frac{r!(n-k+2r)!}{(n-k+r)!}\inner{\varphi_r,\psi_r}_{\omega}, 
        \\
         \inner{L^{n-k}\varphi,L^{n-k}\psi}_{\omega}
        &=
        \sum_r\frac{(n-k+r)!(n-k+2r)!}{r!}\inner{\varphi_r,\psi_r}_{\omega},~\text{and}
        \\
        \inner{L^{n-k-1}\varphi,L^{n-k-1}\psi}_{\omega}
        &=
        \sum_r\frac{(n-k-1+r)!(n-k+2r)!}{(r+1)!}\inner{\varphi_r,\psi_r}_{\omega}.
    \end{align*}
\end{lemma}

\begin{proof}
    Since the primitive decomposition is orthogonal with respect to $\inner{\cdot,\cdot}_\omega$ and $\varphi_r, \psi_r$ are primitive $(k-2r)$-forms, Proposition~\ref{prop:inner_product_Lefschetz} implies that
    \begin{align*}
        \inner{\varphi,\psi}
        =
        \sum_r \inner{L^r\varphi_r,L^r\psi_r}
        =
        \sum_r\frac{r!(n-k+2r)!}{(n-k+r)!}\inner{\varphi_r,\psi_r}
    \end{align*}
    The other inequalities can be obtained using the similar reasoning as
    \[
    L^{n-j}\varphi=\sum_rL^{n-j+r}\varphi_r,~L^{n-j}\psi=\sum_rL^{n-j+r}\psi_r
    \]  are the primitive decompositions of $L^{n-j}\varphi$ and $L^{n-j}\psi$ for any $j\geq 0$.
\end{proof}

 \begin{proposition}
 	\label{prop:comparison_pq}
 	Let $\varphi,\psi$ be smooth $(p,q)$-forms with $k:=p+q<n$ and $p\le q$.
 	Then 
 	\begin{equation*}
 	  (n-k)!^2 \inner{\varphi,\psi}_{\omega}
 	\le
 	\inner{L^{n-k}\varphi,L^{n-k}\psi}_{\omega}
 	\le
  \frac{(n-q)!^2}{p!^2}\inner{\varphi,\psi}_{\omega},
 	\end{equation*}
 	and
 	\begin{equation*}
 		(n-k-1)! (n-k)!
 		\inner{\varphi,\psi}_{\omega}
 		\le
 	\inner{L^{n-k-1}\varphi,L^{n-k-1}\psi}_{\omega}
 		\le
 		\frac{(n-q-1)!(n-q)!}{p!(p+1)!}\inner{\varphi,\psi}_{\omega}.
 	\end{equation*}
 \end{proposition}

\begin{proof}
   Let $\varphi=\sum_{r\ge0}L^r\varphi_r,~ \psi=\sum_{r\ge0}L^r\psi_r$
     be the primitive decompositions of $\varphi$ and $\psi$ so that $\varphi_r$, $\psi_r$ are primitive $(p-r,q-r)$-forms. So $\varphi_r=\psi_r=0$ if $r>p$ and Lemma \ref{lem:pnorm_primitive_decomposition} implies that
    \begin{align*}
        \inner{\varphi,\psi}
        =
        \sum_{0\le r\le p}
         \inner{L^r\varphi_r,L^r\psi_r}
        =
        \sum_{0\le r\le p}
        \frac{r!(n-k+2r)!}{(n-k+r)!} \inner{\varphi_r,\psi_r},
    \end{align*}
    and
    \begin{align*}
        \inner{L^{n-k}\varphi,L^{n-k}\psi}
        =&
	\sum_{0\le r\le p}
        \inner{L^{n-k+r}\varphi_r,L^{n-k+r}\psi_r}
        =
        \sum_{0\le r\le p}
        \frac{(n-k+r)!(n-k+2r)!}{r!}\inner{\varphi_r,\psi_r}\\
         =&\sum_{0\le r\le p}
        \frac{(n-k+r)!^2}{r!^2}
        \cdot
        \frac{r!(n-k+2r)!}{(n-k+r)!}
         \inner{\varphi_r,\psi_r}
    \end{align*}
Then we have
    \begin{equation*}
        (n-k)!^2 \inner{\varphi_r,\psi_r}
        \le
        \inner{L^{n-k}\varphi,L^{n-k}\psi}
        \le
        \frac{(n-k+p)!^2}{p!^2}\inner{\varphi,\psi}=  \frac{(n-q)!^2}{p!^2}\inner{\varphi,\psi}
    \end{equation*}
as desired. It also follows from Lemma \ref{lem:pnorm_primitive_decomposition} that
        \begin{align*}
       \inner{L^{n-k-1}\varphi,L^{n-k-1}\psi}
        =&
        \sum_{0\le r\le p}
        \inner{L^{n-k-1+r}\varphi,L^{n-k-1+r}\psi}
        \\
        =&        
        \sum_{0\le r\le p}
        \frac{(n-k-1+r)!(n-k+2r)!}{(r+1)!}\inner{\varphi_r,\psi_r}\\
        =&\sum_{0\le r\le p}
        \frac{(n-k-1+r)!(n-k+r)!}{r!(r+1)!}
        \cdot
        \frac{r!(n-k+2r)!}{(n-k+r)!}
        \inner{\varphi_r,\psi_r}.
            \end{align*}
Therefore, we conclude that
        \begin{align*}
        (n-k-1)! (n-k)!\inner{\varphi,\psi}
        \le
         \inner{L^{n-k-1}\varphi,L^{n-k-1}\psi}
        \le
        \frac{(n-q-1)!(n-q)!}{p!(p+1)!}\inner{\varphi,\psi}.
    \end{align*}
\end{proof}

\section{Proof of Main Theorems}
\label{sec:proof_of_main_theorems}

A differential form on a smooth manifold is $\emph{simple}$ if it can be expressed as a wedge product of differential 1-forms. 
To start with, we recall the following lemma.
        
    \begin{lemma}[See Section 1.7.5 in \cite{Federer-Book}]
        \label{lemma:norm of wedge product}
        For $\varphi\in\bigwedge^p X$ and $\psi\in\bigwedge^q X$, we have
        \begin{equation*}
            \abs{\varphi\we\psi}_\omega
            \le
            \paren{
                \begin{array}{c}
                    p+q \\ p    
                \end{array}
            }^{1/2}
            \abs{\varphi}_\omega
            \cdot
            \abs{\psi}_\omega.
        \end{equation*}
        If either $\varphi$ or $\psi$ is simple, then
        \begin{equation*}
            \abs{\varphi\we\psi}_\omega
            \le
            \abs{\varphi}_\omega
            \cdot
            \abs{\psi}_\omega.
        \end{equation*}
    \end{lemma}
    
 The following lemma will be important in the proof of Theorem \ref{thm:main_theorem1}.
 \begin{lemma}\label{lemma:k from p,q}
 	If $(X,\omega)$ is a K\"{a}hler manifold, then
 	\begin{equation*}
		\lambda_0^k(X)= \textup{min}\{\lambda^{p,q}_0(X):p+q=k\}
 	\end{equation*}
 	for any integer $k\geq 0$.
	Here, 
 	\begin{equation*}
 	\lambda^{p,q}_0(X)=\textup{inf}\{\iinner{\Delta\varphi,\varphi}:\varphi\in \mathcal{D}^{p,q}(X),\,\norm{\varphi}=1\} .
 	\end{equation*}
 \end{lemma}
 \begin{proof}
 	Let $\varphi\in \mathcal D^k(X)$ with $\norm{\varphi}$=1. Then by \eqref{eqn:pq_decomposition}, there exists a set  $\{\varphi^{p,q}\in \bigwedge\nolimits^{\!p,q}X:p+q=k\}$ of orthogonal forms such that 
 	\[
 	\varphi=\sum_{p+q=k}\varphi^{p,q},~\norm{\varphi}^2=\sum_{p+q=k}\norm{\varphi^{p,q}}^2=1.
 	\]
 	Since $(X,\omega)$ is K\"{a}hler, it follows from (\ref{eq:Kahler_identity}) that each $\Delta \varphi^{p,q}$ is a $(p,q)$-form. So
 	\begin{align*}
 		\iinner{\Delta\varphi,\varphi}=\sum_{p+q=k}\iinner{\Delta\varphi^{p,q},\varphi^{p,q}}
 		\geq \sum_{p+q=k}\lambda_0^{p,q}(X)\norm{\varphi^{p,q}}^2
 		\geq \textup{min}\{\lambda^{p,q}_0(X):p+q=k\}
 	\end{align*}
 	and this implies that $\lambda_0^k(X)\geq \textup{min}\{\lambda^{p,q}_0(X):p+q=k\}.$ Note also that, given a pair $(p,q)$ satisfying $p+q=k$, we have   
 	\[
 	\lambda^{p,q}_0(X)=\textup{inf}\{\iinner{\Delta\varphi^{p,q},\varphi^{p,q}}:\varphi^{p,q}\in \mathcal{D}^{p,q}(X)\norm{\varphi^{p,q}}=1\} \geq \lambda_0^k(X)
 	\]
 	as any $(p,q)$-form is a $k$-form.
 \end{proof}

Now, we present the proofs of our main results, Theorems~\ref{thm:main_theorem1} and \ref{thm:main_theorem2}. 
While the underlying arguments for both theorems are essentially the same, Theorem~\ref{thm:main_theorem2} represents a more specific case and is thus simpler to prove. 
For this reason, we will first prove Theorem~\ref{thm:main_theorem2} and then proceed to the proof for Theorem~\ref{thm:main_theorem1}.

By Theorem \ref{thm:Hodge_theory}~(1), it suffices to settle these theorems for smooth forms with compact support.

\begin{proof}[Proof of Theorem~\ref{thm:main_theorem2}]
Choose  $\varphi\in C_0^\infty(X)$ and let $\psi:=L^n\varphi=\varphi\omega^n\in \Omega^{2n}(X)$.  Then by  Proposition~\ref{prop:inner_product_Lefschetz}, we have
 $\norm{\psi}=n!\norm{\varphi}$ and
\begin{equation*}
    \iinner{\Delta\psi,\psi}
    =
    \iinner{\Delta L^n\varphi,L^n\varphi}
    =
    \iinner{L^n\Delta\varphi,L^n\varphi}
    =
    n!^2\iinner{\Delta\varphi,\varphi}.
\end{equation*}
Note that the Leibniz rule yields
\begin{equation}
\begin{aligned}
    \psi=\varphi\omega^n
    &=
    \varphi d\eta\wedge\omega^{n-1}
    =
    d\paren{\varphi\eta\wedge\omega^{n-1}}
    -
    d\varphi\wedge\eta\wedge\omega^{n-1}=d\theta-\psi',
\end{aligned}
\end{equation}
where $\theta := \varphi\eta\wedge\omega^{n-1}$ and $\psi':=d\varphi\wedge\eta\wedge\omega^{n-1}.$ Since any 1-form is primitive, it follows from Proposition~\ref{prop:inner_product_Lefschetz} and Lemma~\ref{lemma:norm of wedge product} that
\begin{align*}
	\norm{\theta}^2
	&=
	\norm{L^{n-1}(\varphi\eta)}^2
	=
	(n-1)!^2
	\norm{\varphi\eta}^2
	\le
	(n-1)!^2
	\norm{\eta}^2_{L^\infty}
	\norm{\varphi}^2,\\
\norm{\psi'}^2
&\le\norm{\eta}^2_{L^\infty}
\norm{L^{n-1}(d\varphi)}^2\leq (n-1)!^2\norm{\eta}^2_{L^\infty}\norm{d\varphi}^2\leq 
(n-1)!^2\norm{\eta}^2_{L^\infty}
\iinner{\Delta\varphi,\varphi}.
\end{align*}
Finally, the estimates above imply 
\begin{eqnarray*}
n!^2\norm{\varphi}^2
=
\norm{\psi}^2
&=&
\iinner{\psi,d\theta-\psi'}
\le
\abs{\iinner{\psi,d\theta}}
+
\abs{\iinner{\psi,\psi'}}
=
\abs{\iinner{d^*\psi,\theta}}
+
\abs{\iinner{\psi,\psi'}}
\\
&\le&
\iinner{\Delta\psi,\psi}^{1/2}\cdot
\norm{\theta}
+
\norm{\psi}\cdot \norm{\psi'}
\\
&\le&
(n-1)!\norm{\eta}_{L^\infty}
\paren{
\iinner{\Delta\psi,\psi}^{1/2}
\norm{\varphi}
+
\norm{\psi}\iinner{\Delta\varphi,\varphi}^{1/2}
}
\\
&=&
2(n-1)!n!\norm{\eta}_{L^\infty}
\norm{\varphi}\iinner{\Delta\varphi,\varphi}^{1/2}
\end{eqnarray*}
as desired.
\end{proof}

\begin{remark}
The bottom of the spectrum of the Laplace--Beltrami operator on bounded strongly pseudoconvex domains, equipped with K\"ahler metrics induced by defining functions, was computed by Li and Tran~\cite{Li_Tran2010}.
\end{remark}

\begin{proof}[Proof of Theorem~\ref{thm:main_theorem1}]
Choose  $\varphi\in \mathcal{D}^{p,q}(X)$ with $p\leq q,~k:=p+q<n$ and let $\psi:=L^{n-k}\varphi\in \Omega^{2n-k}(X) $. Then 
\[
	(n-k)!^2\norm{\varphi}^2
\le
\norm{\psi}^2
\le
\frac{(n-q)!^2}{p!^2}\norm{\varphi}^2,
\]
and
	\begin{equation*}
		(n-k)!^2\iinner{\Delta\varphi,\varphi}
		\le 
		\iinner{L^{n-k}\Delta \varphi,L^{n-k}\varphi}=	\iinner{\Delta\psi,\psi}
		\le
		\paren{\frac{(n-q)!}{p!}}^2
		\iinner{\Delta\varphi,\varphi}
	\end{equation*}
by Proposition \ref{prop:comparison_pq}. Note also that the Leibniz rule yields
	\begin{align*}
		\psi=\varphi\wedge\omega^{n-k}
		&=
		d\paren{\varphi\wedge\eta\wedge\omega^{n-k-1}}
		-
		d\varphi\wedge\eta\wedge\omega^{n-k-1}
		=d\theta-\psi',
	\end{align*}
	where $\theta:=\varphi\wedge\eta\wedge\omega^{n-k-1}$ and $\psi':=d\varphi\wedge\eta\wedge\omega^{n-k-1}$. Then
	Lemma~\ref{lemma:norm of wedge product} and Proposition~\ref{prop:comparison_pq} imply
	\begin{align*}
			\norm{\theta}^2
		&\le
		\norm{\eta}^2_{L^\infty}
		\norm{L^{n-k-1}\varphi}^2
		\le
		 \frac{(n-q-1)!(n-q)!}{p!(p+1)!}
		\norm{\eta}_{L^\infty}^2
		\norm{\varphi}^2\leq 	\norm{\eta}^2_{L^\infty}\frac{(n-q)!^2}{(p+1)!^2}	\norm{\varphi}^2,\\
		\norm{\psi'}^2
		&		\le
		\norm{\eta}^2_{L^\infty}
		\norm{L^{n-k-1}(d\varphi)}^2=	\norm{\eta}^2_{L^\infty}
		\paren{\norm{L^{n-k-1}(\partial\varphi)}^2+	
			\norm{L^{n-k-1}(\bar{\partial}\varphi)}^2}
		\\
		&\leq
		\norm{\eta}^2_{L^\infty}
		\paren{\frac{(n-q)!^2}{(p+1)!^2}\norm{\partial\varphi}^2+\frac{(n-q-1)!^2}{p!^2}\norm{\bar{\partial}\varphi}^2}\leq \norm{\eta}^2_{L^\infty}\frac{(n-q)!^2}{(p+1)!^2}\norm{d\varphi}^2
		\\
		&\le
		\norm{\eta}^2_{L^\infty}\frac{(n-q)!^2}{(p+1)!^2}\iinner{\Delta\varphi,\varphi}.
	\end{align*} 
	So we obtain
	\begin{eqnarray*}
		(n-k)!^2\norm{\varphi}^2&\le& \norm{\psi}^2
		=
		\iinner{\psi,d\theta-\psi'}	\leq \iinner{\Delta\psi,\psi}^{1/2}
		\norm{\theta}
		+
		\norm{\psi}\norm{\psi'}
		\\
		&\le& \norm{\eta}_{L^\infty} \frac{(n-q)!}{(p+1)!}\paren{
			\iinner{\Delta\psi,\psi}^{1/2}
			\norm{\varphi}
			+
			\norm{\psi}\iinner{\Delta\varphi,\varphi}^{1/2}
		}
		\\
		&\le&2\frac{(n-q)!^2}{(p+1)!^2}\cdot\norm{\eta}_{L^\infty}
		\norm{\varphi}\iinner{\Delta\varphi,\varphi}^{1/2}
	\end{eqnarray*}
	which is the desired estimate. 
	If $k<n$ and $p\geq q$, then we use the estimate above for $\ol{\varphi}$ to obtain the same estimate. If $k>n$, then $\ast \varphi \in \mathcal{D}^{n-p,n-q}(X) $ is a $(2n-k)$-form with $2n-k<n$ so that
	\[
	\iinner{\Delta\ast\varphi, \ast\varphi}=\iinner{\ast\Delta\varphi, \ast\varphi}=\iinner{\Delta\varphi, \varphi}\geq \frac{c_{n-p,n-q}}{\norm{\eta}^2_{L^\infty}}\norm{\ast\varphi}^2=\frac{c_{n-p,n-q}}{\norm{\eta}^2_{L^\infty}}\norm{\varphi}^2
	\] 
	and this settles the estimate for any form in $\mathcal{D}^{p,q}(X)$ with $p+q\neq n$.
	
	Given a real number $x$, we denote by $\rd{x}$ the largest integer less than or equal to $x$ and by $\ru{x}$ the smallest integer greater than or equal to $x$. 
	Let $p,q\geq 0$ be integers satisfying $p\leq q$, $p+q=k< n$, and set $\alpha:=q-p\geq 0$.
	Then
	\begin{equation*}
		k=p+q=2p+\alpha=2q-\alpha,
	\end{equation*}
	and 
	\[
	\ru{k/2}+1=p+1+\ru{\alpha/2}, ~n-\rd{k/2}=n-q-\rd{-\alpha/2}=n-q+\ru{\alpha/2}.
	\]
	Hence we obtain
	\begin{equation*}
	\frac{(\ru{k/2}+1)!}{(n-\rd{k/2})!}=\frac{(p+1+\ru{\alpha/2})!}{(n-q+\ru{\alpha/2})!}\leq \frac{(p+1)!}{(n-q)!}
  \end{equation*}
	since $n-q\geq 1+p$. 
  So it follows that
	\[
	\lambda^{p,q}_0(X)\geq   \frac{(n-k)!^4}{4}
	\cdot
	\frac{(\ru{k/2}+1)!^4}{(n-\rd{k/2})!^4}
	\cdot
	\frac{1}{\norm{\eta}_{L^\infty}^2}.
	\]
	Combining this estimate together with Lemma \ref{lemma:k from p,q} yields the desired estimate for $k<n$. If $k>n$, then use Lemma \ref{lemma:k from p,q} and the Hodge star operator to conclude that
	\[
	\lambda^k_0(X)=\lambda^{2n-k}_0(X)=\textup{min}\{\lambda^{p,q}_0(X):p+q=2n-k\}\geq \frac{c_{2n-k}}{	\norm{\eta}_{L^\infty}^2}.
	\]
	This completes the proof.
\end{proof}

\begin{remark}
\label{rmk:better_estimate}
The sharper estimates for $\theta$ and $\psi'$ in Theorem \ref{thm:main_theorem1} hold when $\varphi\wedge\eta$ and $d\varphi$ are primitive but they are not so in general for a $k$-form $\varphi$ with $k\geq 1$. If the forms are primitive, then one can proceed as above to obtain
	\begin{equation*}
		\norm{\theta}^2
		=
	\norm{L^{n-k-1}(\varphi\wedge\eta)}^2
		\le
		(n-k-1)!^2
		\norm{\eta\wedge\varphi}^2
		=
		(n-k-1)!^2
	\norm{\eta}_{L^\infty}^2
		\norm{\varphi}^2,
	\end{equation*}
	and
	\begin{equation*}
		\begin{aligned}
			\norm{\psi'}^2
			&=
			\norm{L^{n-k-1}(d\varphi\wedge\eta)}^2
			\le
			\norm{\eta}^2_{L^\infty}
			\norm{L^{n-k-1}(d\varphi)}^2
			=
			(n-k-1)!^2\norm{\eta}^2_{L^\infty}
			\iinner{\Delta\varphi,\varphi}
		\end{aligned}
	\end{equation*}
	so that
	\begin{equation*}
		\iinner{\Delta\varphi,\varphi}
		\ge
		\frac{(n-k)^2}{4\norm{\eta}^2_{L^\infty}}\norm{\varphi}^2.
	\end{equation*}
\end{remark}

\section{Middle dimension case}
\label{sec:middle_dimension}
Let $(X,\omega)$ be a K\"{a}hler hyperbolic manifold of complex dimension $n$  where $\omega=d\eta$ with $\norm{\eta}_{L^\infty}<+\infty$. 
If there exists a discrete subgroup $\Gamma$ of the isometry group of $X$ such that the quotient space $X/\Gamma$ is compact, then $\harm^n_{(2)}(X)\neq 0$ by Theorem 2.5 in \cite{Gromov1991}. 
This implies that $\lambda_0^n(X)=0$; however, in this section, we shall use the arguments in \cite{Gromov1991} to derive an explicit uniform lower bound for the spectrum of the Laplacian on the orthogonal complement of $\harm^n_{(2)}(X)$ in $\Ltwo{n}(X)$ for any K\"{a}hler hyperbolic manifold. We start with the following theorem.
\begin{theorem}\label{thm:orthogonal_complement_pq}
	Let $(p,q)$ be a pair of nonnegative integers such that $p+q=n$. If $\varphi\in \Ltwo{p,q}(X)\cap \textup{Dom}(\Delta)$ is orthogonal to $\harm^{p,q}_{(2)}(X)$, then
\[
\iinner{\Delta\varphi,\varphi}
\geq
\frac{\min(c_{p,q-1},c_{p,q+1})}{\norm{\eta}^2_{L^\infty}}
\norm{\varphi}^2.
\]
\end{theorem}

\begin{proof}
Given a pair of integers $(r,s)$ with $r+s\neq n$, we first show that the image of the operator $\bar{\partial}:\Ltwo{r,s}(X)\rightarrow \Ltwo{r,s+1}(X)$ is closed. 
Let $\{\bar{\partial} \varphi_m\}$ be a convergent sequence in $\Ltwo{r,s+1}(X)$ for some $\{\varphi_m\}\subset \Ltwo{r,s}(X)$. 
Since 
\begin{equation*}
	\Ltwo{r,s}(X)=\textup{Ker}(\bar{\partial})\oplus\ol{\textup{Im}(\bar{\del}^*)},
\end{equation*}
one may assume that $\bar{\partial}^{\ast}\varphi_m=0$ for each $m\geq 1$; see p.364 of \cite{Demailly-Book}. Then
\begin{align*}
\norm{\varphi_{\ell}-\varphi_m}^2
&\leq 
\frac{\norm{\eta}^2_{L^\infty}}{c_{r,s}}\iinner{\varphi_{\ell}-\varphi_m,\Delta(\varphi_{\ell}-\varphi_m)}
=
\frac{\norm{\eta}^2_{L^\infty}}{2c_{r,s}}\iinner{\varphi_{\ell}-\varphi_m,\Delta_{\bar{\partial}}(\varphi_{\ell}-\varphi_m)}\\
&\leq 
\frac{\norm{\eta}^2_{L^\infty}}{2c_{r,s}}\cdot \norm{\bar{\partial}\varphi_{\ell}-\bar{\partial}\varphi_{m}}^2
\end{align*}
for any $\ell,m\geq 1$ by Theorem \ref{thm:main_theorem1} and (\ref{eq:Kahler_identity}).
Therefore, $\{\varphi_{m}\}$ also converges in $\Ltwo{r,s}(X)$ and we conclude from the closedness of the graph of $\bar{\del}$ that $\lim_{m\to \infty}\bar{\partial}\varphi_m\in \im\,\bar{\del}$, which says that $\im\,\bar{\del}\subset \Ltwo{r,s+1}(X)$ is closed. 
Similary, one can show that $\im\,\bar{\del}^*\subset \Ltwo{r,s-1}(X)$ is also closed for $r+s\neq n$.
In particular, it follows from Remark~\ref{rmk:Hodge_theory} that for $p+q=n$ we have
\begin{equation}\label{eqn; complex Hodge decomp.}
	\Ltwo{p,q}(X)
	=
	\harm^{p,q}_{(2)}(X)\oplus{\im\,\bar{\del}}\oplus{\im\,\bar{\del}^*}.
\end{equation}

Let $\varphi$ be an element in $\Ltwo{p,q}(X)\cap \textup{Dom}(\Delta)$ orthogonal to $\harm^{p,q}_{(2)}(X)$. 
Then by (\ref{eqn; complex Hodge decomp.}), there exist $\alpha\in	\Ltwo{p,q-1}(X)$ and $\beta\in	\Ltwo{p,q+1}(X)$ such that $\varphi=\bar{\del}\alpha+\bar{\del}^*\beta$.  Using (\ref{eqn; complex Hodge decomp.}) for $(r,s) \in \{(p,q-1),(p,q+1)\}$, one may assume further that $\bar{\del}^*\alpha=0$ and $\bar{\del}\beta=0$. So 
\begin{equation*}
	\iinner{\varphi,\varphi}
	=
	\iinner{\bar{\del}\alpha,\bar{\del}\alpha}+\iinner{\bar{\del}^*\beta,\bar{\del}^*\beta}
	=
	\iinner{\Delta_{\bar{\del}}\alpha,\alpha}+\iinner{\Delta_{\bar{\del}}\beta,\beta},
\end{equation*}
where $\Delta_{\bar{\del}}\alpha=\bar{\del}^*\varphi$ and $\Delta_{\bar{\del}}\beta=\bar{\del}\varphi$. 
By Theorem  \ref{thm:main_theorem1} and the Cauchy-Schwarz inequality, we have
\[
\norm{\alpha}\leq \frac{2\norm{\eta}_{L^\infty}^2}{c_{p,q-1}}\norm{\Delta_{\bar{\del}}\alpha},~\norm{\beta}\leq \frac{2\norm{\eta}_{L^\infty}^2}{c_{p,q+1}}\norm{\Delta_{\bar{\del}}\beta}
\]
so that
\begin{align*}
	\iinner{\Delta_{\bar{\del}}\alpha,\alpha}&\le \norm{\Delta_{\bar{\del}}\alpha}\norm{\alpha}\leq 
	\frac{2\norm{\eta}^2_{L^\infty}}{c_{p,q-1}}
	\iinner{\Delta_{\bar{\del}}\alpha,\Delta_{\dbar}\alpha},\\
		\iinner{\Delta_{\bar{\del}}\beta,\beta}&\le  \norm{\Delta_{\bar{\del}}\beta}\norm{\beta}\le
	\frac{2\norm{\eta}^2_{L^\infty}}{c_{p,q+1}}
	\iinner{\Delta_{\bar{\del}}\beta,\Delta_{\bar{\del}}\beta}.
\end{align*}
Hence we conclude that
\begin{align*}
	\iinner{\varphi,\varphi}
	&=
	\iinner{\Delta_{\dbar}\alpha,\alpha}+\iinner{\Delta_{\dbar}\beta,\beta}
	\\
	&\le
	\frac{2\norm{\eta}^2_{L^\infty}}{\textup{min}(c_{p,q-1},c_{p,q+1})}
	\paren{\iinner{\Delta_{\bar{\del}}\alpha,\Delta_{\bar{\del}}\alpha}+\iinner{\Delta_{\bar{\del}}\beta,\Delta_{\bar{\del}}\beta}}
	\\
	&=
	\frac{2\norm{\eta}^2_{L^\infty}}{\textup{min}(c_{p,q-1},c_{p,q+1})}
	\paren{\iinner{\bar{\del}^*\varphi,\bar{\del}^*\varphi}+\iinner{\bar{\del}\varphi,\bar{\del}\varphi}}
	=
	\frac{\norm{\eta}^2_{L^\infty}}{\textup{min}(c_{p,q-1},c_{p,q+1})}
	\iinner{2\Delta_{\dbar}\varphi,\varphi}\\
	&=	\frac{\norm{\eta}^2_{L^\infty}}{\textup{min}(c_{p,q-1},c_{p,q+1})}
	\iinner{\Delta\varphi,\varphi}
\end{align*}
as desired.
\end{proof}

\begin{theorem}
	If $\varphi\in\dom(\Delta)\cap\Ltwo{n}(X)$ is orthogonal to $\harm^n_{(2)}(X)$, then
	\[
	\norm{\varphi}^2\leq 
	\frac{\norm{\eta}^2_{L^\infty}}{\textup{min}(c_{n-1},c_{n+1})}
	\iinner{\Delta\varphi,\varphi}.
	\]
\end{theorem}

\begin{proof}
Note first that $\varphi=\sum_{p+q=n}\varphi^{p,q}$ where each $\varphi^{p,q}\in \Ltwo{p,q}(X)$ is orthogonal to $\mathcal{H}_2^{p,q}(X)$. Then by Theorem \ref{thm:orthogonal_complement_pq}, we have
\begin{align*}
\norm{\varphi}^2
	=
	\sum_{p+q=k}\norm{\varphi^{p,q}}^2
	&\leq 
	\sum_{p+q=k}\frac{\norm{\eta}^2_{L^\infty}}
	{\textup{min}(c_{p,q-1},c_{p,q+1})}\iinner{\Delta\varphi^{p,q},\varphi^{p,q}}
	\\
	&\leq
	\sum_{p+q=k}\frac{\norm{\eta}^2_{L^\infty}}
	{\textup{min}(c_{n-1},c_{n+1})}\iinner{\Delta\varphi^{p,q},\varphi^{p,q}}
	\\
	&=
	\frac{\norm{\eta}^2_{L^\infty}}{\textup{min}(c_{n-1},c_{n+1})}\iinner{\Delta\varphi,\varphi}
\end{align*}
as $c_{p,q-1}\geq c_{n-1}$ and $c_{p,q+1}\geq c_{n+1}$ for each $(p,q)$ with $p+q=n$.

\end{proof}

\section{Spectrum of bounded symmetric domains}
\label{sec:Spectrum_BSD}

\subsection{K\"ahler hyperbolicity length}
Let $\Omega$ be an irreducible bounded symmetric domain in $\CC^n$, especially a Cartan/Harish-Chandra embedding of the corresponding Hermitian symmetric space of noncompact type, and let $N_\Omega$ be its generic norm. The Bergman kernel $K_\Omega$ of $\Omega$ is  then of the form 
\begin{equation*}
K_\Omega(z,w) 
	= c N_\Omega(z,w)^{-c_\Omega}
\end{equation*}
for a normalizing constant $c$ by the Euclidean volume of $\Omega$ and the genus $c_\Omega$ of $\Omega$ which is a positive integer. 
The K\"ahler-hyperbolicity length $\sL_\Omega$ of $(\Omega,\omega)$ is defined by
\begin{equation*}
\sL_\Omega = \sqrt{rc_\Omega}
\end{equation*}
where $r$ is the rank of $\Omega$, the dimension of a maximal totally geodesic polydisc in $\Omega$. 
If a bounded symmetric domain $\Omega$ is not irreducible, then it is biholomorphic to a product
\begin{equation*}
\Omega\cong\Omega_1\times\cdots\times\Omega_s
\end{equation*}
where each $\Omega_i$ is an irreducible bounded symmetric domain. In this case, we define
\begin{equation*}
\sL_\Omega 
	= \paren{\sum_{j=1}^s \sL_{\Omega_j}^2}^{1/2} 
	=\paren{ \sum_{j=1}^s r_j c_{\Omega_j} }^{1/2}
\end{equation*}
where each $r_j$ and $c_{\Omega_j}$ are the rank and the genus of $\Omega_j$.

\begin{proof}[Proof of Theorem~\ref{thm:main_cor}]
Let $\omega$ be the complete K\"ahler-Einstein metric on $\Omega$ with Ricci curvature $-\sK$.
By the argument of Section 4.1 in~\cite{Choi_Lee_Seo2025Ar}, one can see that the maximal holomorphic sectional curvature of $\Omega$ coincides with $-2\sK/\sL_\Omega^2$; thus we have
\begin{equation*}
\frac{2\sK}{\sL_\Omega^2} \geq K
\end{equation*}
for any upper bound $-K<0$ of holomorphic sectional curvature of $(\Omega,\omega)$. In fact, this is given by the Gaussian curvature of a totally geodesic holomorphic disc of full rank $r = r_1 +\cdots+ r_s$ in a maximal polydisc of $\Omega$.
Moreover, Theorem 1.1 in~\cite{Choi_Lee_Seo2025Ar} implies that there exists a potential $\varphi$ of $\omega$ such that 
\begin{equation}
\label{eqn:constant_gradient_length}
	\abs{d^c\varphi}_\omega^2\equiv\frac{\sL_\Omega^2}{2\sK}.
\end{equation}
This means that $\eta=d^c\varphi$ satisfies $\norm{\eta}^2_{L^\infty}=\sL_\Omega^2/2\sK$.
Together with Theorem~\ref{thm:main_theorem2}, it follows that
\begin{equation*}
	\lambda_0(X)
	\ge
	\frac{n^2}{4\norm{\eta}^2_{L^\infty}}
	=
	\frac{n^2}{4}\cdot\frac{2\sK}{\sL_\Omega^2}
	\ge
        \frac{n^2}{4}\cdot K,
\end{equation*}
This completes the proof.
\end{proof}

We end this subsection with the following theorem concerning the $L^\infty$-norm of any global $1$-form $\eta$ with $d\eta$ being the Bergman metric $\omega$ on bounded symmetric domains.
Note that on bounded symmetric domains, the Bergman metric is the K\"ahler-Einstein metric whose Ricci curvture $-1$.

\begin{theorem}\label{thm:application'}
    Let $\eta$ be a global 1-form of a bounded symmetric domain.
    If $d\eta$ is the K\"ahler form of Bergman metric, then
    \begin{equation*}
        \norm{\eta}^2_{L^\infty}\ge \frac{\sL_\Omega^2}{2}.
    \end{equation*} 
\end{theorem}

This theorem is proved in~\cite{Choi_Lee_Seo2025Ar} under the additional assumption that $\eta$ is $d^c$-closed, which turns out to be superfluous here.

In fact, the problem can be reduced to the unit disc case.
The following lemma says that the $L^\infty$-norm of $\eta$ can be bounded below by the Gaussian curvature of the Poincar\'e metric.

\begin{lemma}\label{lem:disc_case}
Let $(\DD,\omega)$ be the unit disc equipped with the Poincar\'e metric with Gaussian curvature $-\sK$.
If $\eta$ is a global $1$-form with $d\eta=\omega$, then
\begin{equation*}
	\norm{\eta}^2_{L^\infty}\ge\frac{1}{\sK}.
\end{equation*} 
\end{lemma}

\begin{proof}
Suppose to the contrary that there exists a global 1-form $\eta$ on $\DD$ such that
    \begin{equation*}
        d\eta=\omega
        \quad\text{and}\quad
        \norm{\eta}^2_{L^\infty}=c<\frac{1}{\sK}.
    \end{equation*}
Note that $\sK\omega$ has the Gaussian curvature $-1$\ and $\sK\eta$ is a global $1$-form such that $d\paren{\sK\eta}=\sK\omega$.
It follows from Theorem~\ref{thm:main_theorem2} that
\begin{equation*}
	\lambda_0(\DD,\sK\omega)\ge\frac{1}{4c}>\frac{\sK}{4}.
\end{equation*}
On the other hand, it is well known that $\lambda_0(\DD,\sK\omega)=1/4$ (e.g. see~\cite{McKean1970,Chavel-Book}).
Hence we have
\begin{equation*}
	\frac{1}{4}
	=
	\lambda_0(\DD,\sK\omega)
	=
	\frac{1}{\sK}\lambda_0(\DD,\omega)
	>
	\frac{1}{\sK}\cdot\frac{\sK}{4}
	=\frac{1}{4},
\end{equation*}
which yields a contradiction.
\end{proof}

\begin{proof}[Proof of Theorem~\ref{thm:application'}]
    Again suppose to the contrary that there exists a $1$-form $\eta$ such that
    \begin{equation*}
        \norm\eta^2_{L^\infty}<\frac{\sL_\Omega^2}{2}.
    \end{equation*}
    First we consider the irreducible case.
    Let $\Omega$ be an irreducible bounded symmetric domain of rank $r$, which is embedded in $\CC^n$ via the Cartan/Harish-Chandra embedding.
    Then there exists a maximal polydisc of dimension $r$ which is embedded in $\Omega$ of the form 
    $$
    \DD^n=\set{(z^1,\dots,z^r,0,\dots,0)\in\Omega}.
    $$
    and its Bergman kernel $K_{\DD^r}$ satisfies
    \begin{equation*}
        K_\Omega(z,z)=d_1K_{\DD^r}(z,z)^{c_\Omega/2}
        \quad\text{on}\;\;\DD^r.
    \end{equation*}
    For the details, see~\cite{Choi_Lee_Seo2025Ar}. Since $\ric(\omega)=-\omega$, the complete K\"ahler-Einstein metric $\omega_{\DD^r}$ defined by the potential $(c_\Omega/2)\log K_{\DD^r}$ has Ricci curvature $2/c_\Omega$.
    Now we consider a totally geodesic disc $\iota:\DD\hookrightarrow\DD^r\subset\Omega$ of maximal rank defined by $z\rightarrow(z,\dots,z,0,\dots,0)$.
    Then it follows from Proposition 3.2 in \cite{Choi_Lee_Seo2025Ar} that the pull-back metric $\omega_\DD=\iota^*\omega_{\DD^r}$ has the constant Gaussian curvature 
    \begin{equation*}
        \kappa
        =-\frac{2/c_\Omega}{r}
        =-\frac{2}{rc_\Omega}
        =-\frac{2}{\sL^2_\Omega}.
    \end{equation*}
    Since $d(\iota^*\eta)=\omega_\DD$, Lemma~\ref{lem:disc_case} implies that 
    \begin{equation*}
        \norm{\eta\vert_{\DD}}^2_{L^\infty}
        \ge
        \frac{\sL^2_\Omega}{2}.
    \end{equation*}
    Now $\iota:(\DD,\omega_\DD)\to(\Omega,\omega)$ is the isometric embedding, so we have $\norm{\eta}^2_{L^\infty}\geq\norm{\iota^*\eta}^2_{L^\infty}$.
    This completes the proof of irreducible case.
    \bigskip
    
    In the case of a general bounded symmetric domain, the same argument applies. 
		Indeed, any symmetric bounded domain $\Omega$ is decomposed as 
		\begin{equation*}
		\Omega=\Omega_1\times\cdots\times\Omega_s
		\end{equation*}
		with  irreducible factors $\Omega_j$, $j=1,\ldots,s$, a maximal polydisc in $\Omega$ is given by
    \begin{equation*}
        \DD^r = \DD^{r_1}\times\cdots\times\DD^{r_s}
    \end{equation*}
    where each $\DD^{r_j}$ is a maximal polydisc of $\Omega_j$.
    Then we can apply the same arguemt to the totally geodesic disc in $\DD^r$ of maximal rank, which is given by the maximal totally geodesic disc in each $\DD^{r_j}$.
    This yields the same conclusion as before, which completes the proof.
    For the details, see Section 4.1.2 in~\cite{Choi_Lee_Seo2025Ar}.
\end{proof}

\subsection{Bottom of the spectrum on irreducible bounded symmetric domains}
\label{subsec:irreducible_domains}
Irreducible bounded symmetric domains consist of the following 
four classical type domains:
\begin{align*}
\Omega_{p,q}^{\mathrm{I}} &= \left\{ Z\in M^{\mathbb C}(p,q) : I_p -  ZZ^* >0 \right\} \;,\\
\Omega^{\mathrm{II}}_m &= \left\{ Z\in  M^{\mathbb C}(m,m) : I_m - ZZ^*>0, \,\, Z^t = -Z  \right\}\;,\\
\Omega_{m}^{\mathrm{III}} &= \left\{ Z\in M^{\mathbb C}(m,m) : I_m - ZZ^*  >0,\,\, Z^t = Z  \right\}\;,\\
\Omega_m^{\mathrm{IV}} &= \left\{ Z=(z_1,\ldots, z_m)\in {\mathbb C}^m : 
ZZ^* <1 ,\, 0< 1-2 ZZ^*+ \left| ZZ^t \right|^2   \right\}\;,
\end{align*}
and two exceptional type domains:
\begin{equation*} 
\Omega_{16}^{\mathrm{V}}\;, \quad \Omega_{27}^{\mathrm{VI}}\;.
\end{equation*}
Here $M^{\mathbb C}(p,q)$ denotes the set of $p\times q$ complex matrices and $Z^*$ the complex conjugate transpose of the matrix $Z\in M^{\mathbb C}(p,q)$.

Let $\Omega$ be an irreducible bounded symmetric domain in $\CC^n$. The Bergman kernel $K_\Omega$ is of the form 
\begin{equation*}
K_\Omega(z,z) = c N_\Omega(z,z)^{-c_\Omega}
\end{equation*}
for some positive constants $c$, the genus $c_\Omega$ and the generic norm $N_\Omega$. The constant $c_\Omega$, the dimension $n$ and the rank $r$ are given by as follows:

\setlength{\tabcolsep}{10pt}
\renewcommand{\arraystretch}{1.5}
\begin{table}[h]
\label{data for BSDs}
\caption{Invariants}
\begin{tabular}{ c|c|c|c|c|c|c}
$\Omega$& $\Omega_{p,q}^{\mathrm{I}}$&$\Omega_{m,m}^{\mathrm{II}}$& $\Omega^{\mathrm{III}}_{m,m}$&$\Omega_m^{\mathrm{IV}}$  &$\Omega_{16}^{\mathrm{V}}$& $\Omega_{27}^{\mathrm{VI}}$\\\hline
$c_\Omega$&$p+q$&$2(m-1)$&$m+1$&$m$&$12$&$18$\\\hline
$n$&$pq$&$\frac{m(m-1)}{2}$&$\frac{m(m+1)}{2}$&$m$&$16$&$27$\\\hline
$r$&$p$&$\left[ \frac{m}{2}\right]$&$m$&$2$&$2$&$3$
\end{tabular}
\end{table}

As we already know, the Bergman metric $\omega_\Omega=d\dc \log K_\Omega$ coincide the complete K\"ahler-Einstein metric with Ricci curvature $-1$.
Then \eqref{eqn:constant_gradient_length} implies that there exists a global potential $\varphi:\Omega\rightarrow\RR$ such that
\begin{equation*}
	\abs{d^c\varphi}_\omega^2=\frac{\sL_\Omega^2}{2}.
\end{equation*}
Hence Theorem~\ref{thm:main_theorem2} gives a lower bounds for all irreducible bounded symmetric domains in terms of the K\"ahler hyperbolicity length.
And the lower bounds for four classical types are given as follows:
\begin{align*}
    &\lambda_0(\Omega^{\mathrm I}_{p,q})\ge\frac{(pq)^2}{2p(p+q)},
    \quad
    \lambda_0(\Omega^{\mathrm{II}}_{m,m})\ge\frac{m^2(m-1)}{16\bparen{m/2}},
    \\
    &\lambda_0(\Omega^{\mathrm{III}}_{m,m})\ge\frac{m(m+1)}{8},
    \quad\text{and}\quad
    \lambda_0(\Omega^{\mathrm{IV}}_m)\ge\frac{m}{4}.
\end{align*}
The lower bounds for two exceptional types are as follows:
\begin{equation*}
    \lambda_0(\Omega^{\mathrm{V}}_{16})\ge\frac{16}{3},
    \quad\text{and}\quad
    \lambda_0(\Omega^{\mathrm{VI}}_{27})\ge\frac{27}{4}.
\end{equation*}
Sharp lower and upper bounds for $\lambda_0$ in case of $\Omega^{\mathrm{II}}_{m,m}$, $\Omega^{\mathrm{III}}_{m,m}$, and $\Omega^{\mathrm{IV}}_{m}$ were obtained in \cite{Long_Li2019,Long_Zhang_Lin_Shen2023}.
The lower bound which they obtained is the same as the above.
It is important to note that their normalization differs from the one used here.
For example, in \cite{Long_Li2019}, the following global potential of the metric is used.
\begin{equation*}
    u=\frac{1}{m+1}\log K(z,z),
\end{equation*}
which implies that $\ric(dd^c u)=-(m+1)dd^c u$. 
Also they used the Laplacian which is two times of $\Delta$.
This leads to the following question:
\begin{question}
Can one obtain a sharp upper bound for $\lambda_0$ on K\"ahler hyperbolic manifolds, in the sense that it coincides with the known lower bound in the case of bounded symmetric domains?
\end{question}



\end{document}